\documentclass[a4paper]{amsart}

\usepackage[english]{babel}
\usepackage{graphicx,subfigure}
\usepackage{amsmath, amsfonts, amsthm, upgreek,latexsym,amssymb}
\usepackage{enumerate}

\def\zent#1{{\bf Z}(#1)}
\def\zentf(#1){{\bf Z}(#1)}

\newenvironment{enumeratei}{\begin{enumerate}[\upshape (a)]}
    {\end{enumerate}}

\newtheorem{theorem}{Theorem}
\newtheorem{lemma}[theorem]{Lemma}

\newtheorem{defn}[theorem]{Definition}

\newtheorem*{thm}{Theorem}
\def\cent#1#2{{\bf C}_{#1}(#2)}

\newcommand{\gamg}{\Gamma(G)}

\begin{document}

\title[Conjugacy classes of finite groups and graph regularity]{Conjugacy classes of finite groups and\\ graph regularity}
\author[M. Bianchi]{Mariagrazia Bianchi}
\address{Mariagrazia Bianchi, Dipartimento di Matematica F. Enriques,
\newline Universit\`a degli Studi di Milano, via Saldini 50,
20133 Milano, Italy.} \email{mariagrazia.bianchi@unimi.it}

\author[R.D. Camina]{Rachel D. Camina}
\address{Rachel D. Camina, Fitzwilliam College, Cambridge, CB3 0DG,
UK.} \email{rdc26@dpmms.cam.ac.uk}

\author[M. Herzog]{Marcel Herzog}
\address{Marcel Herzog, Schoool of Mathematical Sciences, \newline Raymond and
Beverly Sackler Faculty of Exact Sciences,\newline Tel-Aviv
University, Tel-Aviv, 69978, Israel.} \email{herzogm@post.tau.ac.il}

\author[E. Pacifici]{Emanuele Pacifici}
\address{Emanuele Pacifici, Dipartimento di Matematica F. Enriques,
\newline Universit\`a degli Studi di Milano, via Saldini 50,
20133 Milano, Italy.} \email{emanuele.pacifici@unimi.it}

\thanks{The first and the fourth author are partially supported by the
MIUR project ``Teoria dei gruppi e applicazioni''. The second and the third author are grateful to the Department of Mathematics
of the University of Milano for its hospitality and support,
while this investigation was carried out.}

\dedicatory{Dedicated to the memory of David Chillag}

%\date\today

\begin{abstract}
Given a finite group \(G\), denote by \(\gamg\) the simple
undirected graph whose vertices are the distinct sizes of noncentral
conjugacy classes of \(G\), and set two vertices of \(\gamg\) to be
adjacent if and only if they are not coprime numbers. In this note
we prove that, if \(\gamg\) is a \(k\)-regular graph with \(k\geq
1\), then \(\gamg\) is a complete graph with \(k+1\) vertices.

\end{abstract}

\maketitle

\section{Introduction}
Given a finite group $G$, let $\Gamma(G)$ be the simple
undirected graph whose vertices are the distinct sizes of
\emph{noncentral} conjugacy classes of $G$, two of them being
adjacent if and only if they are not coprime numbers. The interplay
between certain properties of this graph and the group structure of
\(G\) has been widely studied in the past decades, and it is
nowadays a classical topic in finite group theory (see, for
instance, \cite{CCs}). The present note is a contribution in this
direction.

In \cite{BHPS} it is conjectured that, for every integer \(k\geq 1\),
the graph $\Gamma(G)$ is $k$-regular if and only if $\Gamma(G)$ is a complete
graph with \(k+1\) vertices. That paper settles the case \(k\leq
3\), whereas in this note, using a different approach, we provide an
affirmative answer to the conjecture in full generality.

\begin{thm}
Let \(G\) be a finite group, and assume that \(\Gamma(G)\) is a
\(k\)-regular graph with \(k\geq 1\). Then \(\Gamma(G)\) is a
complete graph with \(k+1\) vertices.
\end{thm}

Another graph related to the conjugacy classes of finite groups,
that has been extensively studied in the literature, is the 
\emph{prime graph} \(\Delta(G)\): in this case the vertices are the
primes dividing some class size of \(G\), and two distinct vertices
\(p\), \(q\) are adjacent if and only if there exists a class size
of \(G\) that is divisible by \(pq\). It is well known that the
graphs \(\Gamma(G)\) and \(\Delta(G)\) share some relevant
properties (for instance, they have the same number of connected
components, and the diameters of the two graphs differ by at most
\(1\)). Nevertheless we remark that, in contrast to the situation
described by the main result of this note, the class of finite
groups \(G\) such that \(\Delta(G)\) is a non-complete, connected
and regular graph is not empty. Such groups have been classified in
Therorem~D of \cite{CDPS}.

\section{The results}
Every group considered in the following discussion is tacitly
assumed to be a finite group. We start by introducing a definition.

\begin{defn}
\label{Partners} Let \(\Gamma\) be a graph and, for a given vertex
\(X\) of \(\Gamma\), denote by \(\nu(X)\) the set of neighbors of
\(X\) in \(\Gamma\) (i.e., the set of vertices of \(\Gamma\) that
are adjacent to \(X\)). Let \(A\), \(B\) be vertices of \(\Gamma\):
we say that \(A\) and \(B\) are \emph{partners} if
\(\nu(A)\cup\{A\}=\nu(B)\cup\{B\}\). This clearly defines an
equivalence relation on the vertex set of $\Gamma$.
\end{defn}

Next, given a group \(G\), we focus on the common divisor graph \(\Gamma(G)\) built on the set of conjugacy class sizes of \(G\), as defined in the Introduction. For \(g\in G\), the \(G\)-conjugacy class of \(g\) will be denoted by \(g^G\).

\begin{lemma}
Let \(G\) be a group, and assume that \(\Gamma(G)\) is a
\(k\)-regular graph with \(k\geq 1\). Then \(\Gamma(G)\) is a
connected graph. \label{Connected}
\end{lemma}

\begin{proof} If \(\Gamma(G)\) is not connected, then by \cite{BHM} (or \cite{K}) it consists
of two isolated vertices, and hence it is not \(k\)-regular for any $k\ge 1$.
\end{proof}

\begin{lemma}
\label{pPower} Let \(G\) be a group, and assume that \(\Gamma(G)\)
is a \(k\)-regular graph with \(k\geq 1\). If there exists \(x\in
G\setminus\zent G\) such that \(|x^G|\) is a prime power, then
\(\Gamma(G)\) is a complete graph.
\end{lemma}

\begin{proof}
Assume that \(|x^G|\) is a power of the prime \(p\). Then the \(k\)
neighbors of \(|x^G|\) are all divisible by \(p\), and the subgraph
of \(\Gamma(G)\) induced by them is complete. If the whole
\(\Gamma(G)\) is not complete, the connectedness of \(\Gamma(G)\) (ensured by Lemma~\ref{Connected}) 
implies the existence of \(y\in G\setminus \zent G\) such that
\(|y^G|\) is not divisible by \(p\) but is adjacent to a neighbor of
\(|x^G|\); this neighbor will have now valency at least \(k+1\),
thus violating the regularity of \(\Gamma(G)\).
\end{proof}

\begin{lemma}
\label{CoprimeOrdersCommuting} Let \(G\) be a group, and assume that
\(\Gamma(G)\) is a regular graph. If \(x\) and \(y\) are noncentral
elements of \(G\) such that \(\cent G x\leq\cent G y\), then
\(|x^G|\) and \(|y^G|\) are partners in \(\Gamma(G)\). In
particular, the following conclusions hold.
\begin{enumeratei}
\item If \(x\) and \(y\) are noncentral elements of \(G\) having coprime orders, and such that \(xy=yx\),
then \(|x^G|\) and \(|y^G|\) are partners in \(\Gamma(G)\).
\item If \(x\) is an element of \(G\) and \(k\) is an integer such that \(x^k\) is noncentral,
then \(|x^G|\) and \(|(x^k)^G|\) are partners in \(\Gamma(G)\).
\end{enumeratei}
\end{lemma}
\begin{proof}
As \(\cent G x\leq \cent G y\), we have that \(|y^G|\) is a divisor
(different from \(1\)) of \(|x^G|\). Clearly we have
\[\nu(|y^G|)\cup\{|y^G|\}\subseteq\nu(|x^G|)\cup\{|x^G|\};\] but the
regularity of \(\Gamma(G)\) forces those two sets to have the same
cardinality. Hence equality holds, and \(|x^G|\) and \(|y^G|\) are
partners. Now, since \(\cent G x\leq \cent G {x^k}\), Conclusion (b)
follows. Moreover, in the setting of Conclusion (a), we have \(\cent
G {xy}=\cent G x\cap\cent G y\); thus, taking into account the
transitivity of being partners, Conclusion (a) follows as well.
\end{proof}

Recall that, for every group \(G\), the diameter of the graph
\(\Gamma(G)\) is at most~\(3\) (see~\cite{CHM} or \cite{K}). The
following proposition shows that no graph of this kind can be
regular of diameter \(3\).

\begin{lemma}
Let \(G\) be a group, and assume that \(\Gamma(G)\) is a
\(k\)-regular graph with \(k\geq 1\). Then the diameter of
\(\Gamma(G)\) is at most \(2\). \label{Diameter2}
\end{lemma}

\begin{proof} The groups \(G\) such that the diameter of \(\Gamma(G)\) is \(3\) are classified in \cite{K}:
they are direct products \(F\times H\) where \((|F|,|H|)=1\), the
graph \(\Gamma(F)\) consists of two isolated vertices, and
\(\Gamma(H)\) is not the empty graph. Now, let \(|x^F|\) be a vertex
of \(\Gamma(F)\), \(|y^H|\) a vertex of \(\Gamma(H)\), and consider
the vertices \(|y^G|\), \(|(xy)^G|\) of \(\Gamma(G)\). Since \(\cent
G{xy}\leq\cent G y\) and \(\Gamma(G)\) is regular, it follows by
Lemma~\ref{CoprimeOrdersCommuting} that \(|y^G|\) and \(|(xy)^G|\)
are partners. This is a contradiction, because \(|x^G|\) is adjacent
to \(|(xy)^G|\) but not to \(|y^G|\).
\end{proof}

Before proving the main result, it will be convenient to introduce some more terminology. If \(g\in G\) is such that, for
every \(h\in G\), the condition \(\cent G h\leq\cent G g\) implies
\(\cent G h=\cent G g\), then we say that \(\cent G g\) is a
\emph{minimal centralizer} in \(G\). Also, we say that \(g\in G\) is
\emph{strongly noncentral} if the order of \(g\zent G\) (as an
element of \(G/\zent G\)) is not a prime power. Finally, given a prime \(p\), we denote by \(g_p\) and \(g_{p'}\)
respectively the \emph{\(p\)-part} and the \emph{\(p'\)-part} of
\(g\), i.e., \(g_p\) is a \(p\)-element, \(g_{p'}\) a
\(p'\)-element, and \(g_pg_{p'}=g=g_{p'}g_p\). Recall
that both \(g_p\) and \(g_{p'}\) are powers of \(g\). Note also that a prime \(p\) divides \({\textnormal{o}}(g\zent G)\) if and only if \(g_p\not\in\zent G\). 

We are now in a position to prove the main theorem of this note, that was stated in the Introduction.

\begin{proof}[Proof of the Theorem]
Let \(G\) be a group with centre $Z$ and, aiming at a contradiction, assume
that \(\Gamma(G)\) is a non-complete \(k\)-regular graph with
\(k\geq 1\). We start by proving two claims.

\medskip
\noindent{\bf{Claim 1.}} \emph{For every minimal centralizer \(\cent G g\) of
\(G\), there exists a strongly noncentral element \(x\in G\) such that \(\cent
G x=\cent G g\).}

\smallskip
Denoting by \(p\) a prime divisor of the order of
\(gZ\in G/Z\), observe first that \(\cent G g/Z\) cannot be a
\(p\)-group. In fact, if  \(|\cent G g/Z|=p^{n}\), then \(|g^G|\) is
a multiple of all the primes dividing some conjugacy class size of
\(G\), except perhaps \(p\). As \(\Gamma(G)\) is non-complete and regular, there
exists \(y\in G\setminus Z\) such that \(|y^G|\) is coprime to
\(|g^G|\); in other words, \(|y^G|\) is a \(p\)-power, and
Lemma~\ref{pPower} yields a contradiction.

Now, let \(q\neq p\) be a prime divisor of \(|\cent G g/Z|\), and
let \(yZ\in \cent G g/Z\) be an element of order \(q\). We can
clearly assume \(g_{p'}\in Z\), thus we have \(\cent G{g_p}=\cent
G{g}\). Moreover, \(g_p\) and \(y_q\) commute, whence, setting
\(x=g_py_q\), we get \[\cent G x=\cent G{g_py_q}=\cent
G{g_p}\cap\cent G{y_q}=\cent G g\cap\cent G{y_q}.\] The minimality
of \(\cent G g\) forces \(\cent G{x}=\cent G g\), and \(x\) is strongly noncentral in $G$, as desired.

\medskip
\noindent{\bf{Claim 2.}} \emph{There exist two elements \(x\), \(y\) such that \(|x^G|\)
and \(|y^G|\) are coprime, and both \(\cent G x\), \(\cent G y\) are
minimal centralizers.}

\smallskip
Choose \(x\in G\) so that \(\cent G x\) is minimal. As \(\Gamma(G)\)
is non-complete and regular, we can find \(w \in G\setminus Z\) such that
\(|x^G|\) and \(|w^G|\) are coprime. Now, let \(y\in G\) be such
that \(\cent G y\) is minimal and contained in \(\cent G w\).
Lemma~\ref{CoprimeOrdersCommuting} yields that \(|y^G|\) and
\(|w^G|\) are partners and, since \(|x^G|\) and \(|w^G|\) are
coprime, it follows that also \(|x^G|\) and \(|y^G|\) are coprime,
as required.

\medskip
In view of the previous claims, we can now conclude the proof. Let us choose two elements \(x\) and \(y\) as in Claim~2, i.e., \(|x^G|\) and \(|y^G|\) are coprime,
and both \(\cent G x\), \(\cent G y\) are minimal centralizers. In
particular, by Claim~1, we can assume
that both \(x\) and \(y\) are strongly noncentral in $G$.

Since, by Lemma~\ref{Diameter2}, the diameter of \(\Gamma(G)\)
is \(2\), we can find \(w\in G\) such that \(|w^G|\) is a common
neighbor for \(|x^G|\) and \(|y^G|\). Also,
let \(p\) be a prime divisor of \({\textnormal{o}}(wZ)\in G/Z\). As
\(|x^G|\) and \(|y^G|\) are coprime, we may assume that \(\cent G
x\) contains a Sylow \(p\)-subgroup of \(G\); thus, up to replacing
\(x\) by a conjugate of it, we have \(w_p\in\cent G x\). Now, since
\(x\) is strongly noncentral, there exists a prime \(q\neq p\) which
divides the order of \(xZ\); Lemma~\ref{CoprimeOrdersCommuting}
implies that \(|x^G|\) and \(|(x_q)^G|\) are partners, but
\(|(x_q)^G|\) is also a partner of \(|(w_p)^G|\), and the latter is
a partner of \(|w^G|\). As a consequence, \(|x^G|\) and \(|w^G|\)
are partners: this is the final contradiction, as \(|y^G|\) is
adjacent to \(|w^G|\) but not to \(|x^G|\).
\end{proof}

\pagebreak

\end{document}